 \documentclass[a4paper,13pt,abstracton]{article}
 \usepackage[utf8]{inputenc}
   \providecommand{\keywords}[1]{\textbf{\textit{Key words:}} #1}
      \usepackage{amsmath,amsthm,verbatim,amssymb,amsfonts,amscd, graphicx,mathrsfs}
   \usepackage{stmaryrd}
 \usepackage{geometry}
 \usepackage{tabularx}
 \usepackage{diagbox}
 \numberwithin{equation}{section}
 \newtheorem{thm}{Theorem}[section]
 \newtheorem{lem}[thm]{Lemma}
\usepackage{stmaryrd}
 \newtheorem{define}[thm]{Definition}
 \newtheorem{cor}[thm]{Corollary}
 \newtheorem{prop}[thm]{Proposition}
 \newtheorem{rmk}[thm]{Remark}

 \newtheorem{ex}[thm]{Example}
  
 \usepackage{enumerate}
 \usepackage{multicol}
 \usepackage{multirow}
 \usepackage{float}
 \usepackage{listings}
 \usepackage[square,comma,super,numbers,sort&compress]{natbib}
 \usepackage{bm}
  \usepackage{hyperref}
 \hypersetup{colorlinks=true, urlcolor=Cerulean, citecolor=red}
 \usepackage[T1]{fontenc}
 \usepackage{nameref,cleveref}
 \usepackage{nicefrac,xfrac}
 \usepackage[all]{xy}
 \usepackage{color}
\usepackage[section]{placeins}
\usepackage{indentfirst}

\begin{document}
\title{\textbf {Homogeneous ACM bundles on exceptional Grassmannians}}

\author{Xinyi Fang
\thanks{Department of Mathematics, Nanjing University, No. 22, Hankou Road,
Nanjing, 210093, P. R. China,
xyfang@nju.edu.cn.
The research is sponsored by Excellent Postdoctoral Plan of Jiangsu Province.
},
Yusuke Nakayama
\thanks{School of Fundamental Science and Engineering, Waseda University, 3-4-1, Okubo, Shinjuku, Tokyo, 169-8555, yusuke216144@akane.waseda.jp.}
and Peng Ren
\thanks{Shanghai Center for Mathematical Sciences, Fudan University, 2005 Songhu Road, Shanghai, 200438, P. R. China, pren@fudan.edu.cn
%The Research is Sponsored by Innovation Action Plan (Basic research projects) of Science and Technology Commission of Shanghai Municipality (Grant No. 21JC1401900) and Science and Technology Commission of Shanghai Municipality (Grant No. 18dz2271000).
The first author is sponsored by Innovation Action Plan (Basic research projects) of Science and Technology Commission of Shanghai Municipality (Grant No. 21JC1401900) and Excellent Postdoctoral Plan of Jiangsu Province.
}
}

\date{}
\maketitle

%----------------------------------------------------------------------------------------
%	ESSAY BODY
%----------------------------------------------------------------------------------------
%%%%%%%%%%%%%%%%%%%%%%%%%%%%%%%%%%%%%%%%%%%%%%%%%%%%%%%%%%%%%%%%%%%%%%%%%%%%%%%%%%%%%%%%%%%%%%%%%%%%%%%%%%%%%%%%%%%%%%%%%%%%%%%%%%%%%%%%%%%%%%%%%%%%%%%%%%%%%%% ???\UTF{00E8}??%%%%%%%%%%%%%%%%%%%%%%%%%%%%%%%%%%

\begin{abstract}
In this paper, we characterize homogeneous arithmetically Cohen-Macaulay (ACM) bundles over exceptional Grassmannians in terms of their associated data. We show that there are only finitely many irreducible homogeneous ACM bundles by twisting line bundles over exceptional Grassmannians. As a consequence, we prove that some exceptional  Grassmannians are of wild representation type.
\end{abstract}



%%%%%%%%%%%%%%%%%%%%%%%%%%%%%%%%%%%%%%%%%%%%%%%%%%%%%%%%%%%%%%%%%%%%%%%%%%%%%%%%%%%%%%%%%%%%%%%%%%%%%%%%%%%%%%%%%%%%%%%%%%%%%%%%%%%%%%%%%%%%%%%%%%%%%%%%%%%%%%% ??????\UTF{00E8}??%%%%%%%%%%%%%%%%%%%%%%%%%%%%%%%%%%
\section{Introduction}
Vector bundles over projective varieties have been studied for many years. For instance, Horrocks \cite{horrocks1964vector} showed that vector bundles on a projective space over a field of characteristic zero split as the direct sum of line bundles if and only if they have no intermediate cohomology. Since this result was established, research on indecomposable bundles without intermediate cohomology on projective varieties has garnered considerable attention. Such bundles are called arithmetically Cohen--Macaulay (ACM) bundles, which have have been studied extensively. ACM bundles correspond to maximal Cohen--Macaulay  modules over the associated graded ring.

ACM bundles have been studied over  particular varieties. The first nontrivial case involves two-dimensional varieties. For example, Casanellas--Hartshorne \cite{Casanellas2011} proved the existence of stable ACM bundles of arbitrary rank on smooth cubic surfaces. This was the first example of indecomposable ACM bundles of arbitrarily high rank on varieties except curves. Various other studies have also been conducted on this case (see \cite{Ballico2021}, \cite{Notari2017}, \cite{Faenzi2008}, \cite{Watanabe2008}, \cite{Yoshioka2021}). In the case of three -dimensional varieties, Casnati--Faenzi--Malaspina \cite{Casnati2015} classified all rank-two indecomposable ACM bundles on the del Pezzo threefold with Picard number three. In addition, Filip \cite{Filip2014} classified rank-two indecomposable ACM bundles on the general complete intersection of Calabi--Yau threefolds. Other studies have also been conducted on this topic (see \cite{Brambilla2011} and \cite{Ravindra2019}).

The problem of classifying ACM bundles has also been studied on homogeneous varieties. Using the Borel--Bott--Weil Theorem, Costa--Mir\'{o}-Roig \cite{costa2016homogeneous} classified irreducible homogeneous ACM bundles on Grassmannian varieties. Recently, such bundles on isotropic Grassmannians of types $B$, $C$ and $D$ were classified by Du--Fang--Ren \cite{du2022homogeneous}.

The aim of this study is to classify all irreducible homogeneous ACM bundles over homogeneous varieties $X=G/P$, where $G$ is a semi-simple linear algebraic group and $P$ is a maximal parabolic subgroup. This is a generalization of the work of Costa--Mir\'{o}-Roig and Du--Fang--Ren. We derive the necessary and sufficient conditions for an irreducible homogeneous bundle on a homogeneous variety $X$ to be an ACM bundle. The result indicates that only finitely many irreducible homogeneous ACM bundles up to twisting line bundles exist over homogeneous varieties.  In addition, we derive the conditions for the highest weight of an irreducible homogeneous vector bundle on an exceptional Grassmannian to be an ACM bundle. Moreover, we show that some exceptional Grassmannians are of wild representation type. Here, it would be appropriate to mention that there is another interesting class called Ulrich bundles. Such bundles on a homogeneous variety were studied by \cite{costa2015gl}, \cite{fonarev2016irreducible} and \cite{lee2021equivariant}.

\subsection{Statement of results}
Let $G$ be a semi-simple linear algebraic group over the complex field and $P({\alpha_{k}})$ be the maximal parabolic subgroup associated to the simple root, $\alpha_{k}$. A vector bundle, $E$, over $G/P({\alpha_{k}})$ is {\it homogeneous} if the action of $G$ over $G/P({\alpha_{k}})$ can be lifted to $E$. This can be represented by $G\times_{\upsilon}V$, where $\upsilon:P({\alpha_{k}})\to GL(V)$ is a representation of $P({\alpha_{k}})$. If this representation is irreducible, we call $E$ an {\it irreducible\ homogeneous\ bundle\/}. We use $E_{\lambda}$ to denote the irreducible homogeneous vector bundle arising from the irreducible representation of $P({\alpha_{k}})$ with highest weight, $\lambda$. We define set $\Phi_{k,G}^{+}$ as follows:
$$\Phi_{k,G}^{+}:=\{\alpha\in\Phi^{+}_{G}\ |\ (\lambda_{k},\alpha)\neq0\},$$
where $\Phi_{G}^{+}$ is the set of positive roots, $\lambda_{k}$ is the $k$-th fundamental weight and $(,)$ denotes the Killing form. Moreover, for any irreducible homogeneous vector bundle, $E_{\lambda}$, on $G/P({\alpha_{k}})$  with highest weight, $\lambda$, we define its {\it associated datum} $T^{G}_{\lambda,k}$ as follows:
$$T^{G}_{\lambda,k}:=\left\{\frac{(\lambda+\rho,\alpha)}{(\lambda_k,\alpha)} \ |\ \alpha\in\Phi_{k,G}^{+}\right\},$$
where $\rho$ is the sum of all fundamental weights.
Note that $T^{G}_{\lambda,k}$ essentially corresponds to the $step\ matrix$ in \cite{costa2016homogeneous} and \cite{du2022homogeneous}. We now state the main result of this study.
\begin{thm}\label{Thm} Let $E_\lambda$ be an initialized irreducible homogeneous vector bundle with highest weight $\lambda$ over $G/P(\alpha_k)$. Let $T_{k,\lambda}^G $ be its associated datum.
Denote $n_l:=\#\{t\in T_{k,\lambda}^G|t=l \}$. Then $E_\lambda$ is an ACM bundle if and only if $n_l\geq1$ for any integer $l\in [1,M^G_{k,\lambda}],$
where $M_{k,\lambda}^G= \max\{t|t\in T_{k,\lambda}^G\}$.
\end{thm}
This result was obtained by Costa and Mir\'{o}-Roig for the case in which $G/P({\alpha_{k}})$ is an Grassmannian. Moreover, when $G/P({\alpha_{k}})$ is an isotropic Grassmannian of type $B,\ C$ or $D$ , this result was obtained by Du, Fang and Ren. Therefore, the novelty of the result lies in proving the statement over other varieties of exceptional types (i.e., exceptional Grassmannians).  In this paper, we present a unified proof. From this theorem, we can obtain the following corollary.
\begin{cor}\label{Cor} Only finitely many irreducible homogeneous ACM bundles up to tensoring a line bundle exist on $G/P({\alpha_{k}})$.
\end{cor}
In addition, based on the aforementioned theorem, we use these irreducible homogeneous vector bundles as a building block to construct families of indecomposable ACM bundles on some exceptional Grassmannians of arbitrary high rank and dimension (see Section $4$).

$\mathbf{Acknowledgment:}$ We are grateful to Takeshi Ikeda for their helpful advice and comments, to Hajime Kaji for beneficial comments, and to Satoshi Naito and Takafumi Kouno for useful advice. We would like to thank Editage (www.editage.com) for English language editing.
 \paragraph{Notation and convention}

\begin{itemize}
\item $E_n$: the simple Lie group with Dynkin diagram $E_n~(n=6,7,8)$;
\item  $F_4$: the simple Lie group with Dynkin diagram $F_4$;
\item  $G_2$: the simple Lie group with Dynkin diagram $G_2$;
\item  $\alpha_i$: the $i$-th simple roots;

\item  $(a_1,\dots,a_n)$:
$a_1\alpha_1+\dots+a_n\alpha_n$;

\item  $\Phi^+_G$: the set of positive roots of $G$;

\item  $\lambda_k$: the $k$-th fundamental weight;

\item  $\rho$: $\lambda_1+\dots+\lambda_n$
\item  $(\cdot,\cdot)$: the Killing form;
\item $E_\lambda$: the irreducible homogeneous vector bundle with highest weight $\lambda$;
\item $G/P(\alpha_k)$: the homogeneous variety with semisimple complex Lie group $G$ and parabolic subgroup $P(\alpha_k)$;
\item $T_{\lambda,k}^{G}$: the associated datum of $E_\lambda$ on $G/P(\alpha_k)$.
\item $E^\vee$ : the dual of $E$;
\item $h^i(X,E)$: the dimension of $i$-th sheaf cohomology of $E$.
\end{itemize}
\section{Preliminaries}
We begin this section by reviewing some definitions and introducing some notations. All algebraic varieties in this study are defined over the field of complex numbers, $\mathbb{C}$.

\subsection{Exceptional Grassmannians}
Let $G$ be a semisimple complex Lie group and $H$ be a fixed maximal torus of $G$. Denote their Lie algebras by
$\mathfrak{g}$ and $\mathfrak{h}$ respectively. Let $\Phi$ be its root system and $\Delta=\{\alpha_1,...,\alpha_n\}\subset\Phi$ be a set of fixed simple roots.

 Let $I\subset\Delta$ be a subset of simple roots. Define \[\Phi^-(I):=\{\alpha\in\Phi^-|\alpha=\sum\limits_{\alpha_{i}\notin I}p_i\alpha_i\}.\]
Let \[\mathcal{P}(I):=\mathfrak{h}\bigoplus (\oplus_{\alpha\in\Phi^+}\mathfrak{g}_\alpha)\bigoplus (\oplus_{\alpha\in\Phi^-(I)}\mathfrak{g}_\alpha)\]
and $P(I)$ be the subgroup of $G$ such that the Lie algebra of $P(I)$ is $\mathcal{P}(I).$
Recall that a closed subgroup $P$ of $G$ is called \emph{parabolic} if the quotient space $G/P$ is projective.
We have the following theorem to describe all parabolic subgroups of $G$.

\begin{thm}(see \cite{ottaviani1995rational} Theorem 7.8)
Let $G$ be a semisimple simply connected Lie group and $P$ be a parabolic subgroup of G. Then there exists $g\in G$ and $I\subset \Delta$ such that \[g^{-1}Pg=P(I).\]
\end{thm}
From this classification theorem, we always use $P(I)$ to denote the parabolic subgroup of $G$. It's not hard to see that $P(I)$ is a maximal parabolic subgroup of $G$ when $\#|I|=1$. In this paper, we denote $P({\alpha_{k}})$ by the maximal parabolic subgroup associated to the simple root, $\alpha_{k}$.
It's known that when $G$ is a simple Lie group of type $A$, $G/P({\alpha_{k}})$ is the usual Grassmannian. We call $G/P({\alpha_{k}})$ an \emph{isotropic Grassmannian} if $G$ is of type $B$, $C$ or $D$. We call $G/P({\alpha_{k}})$ an \emph{exceptional Grassmannian} if $G$ is of type $E_n~(n=6,7,8)$, $F_4$ or $G_2$.

  In this paper, we focus on the exceptional  Grassmannians $G/P({\alpha_{k}})$  where the Dynkin diagrams of $G$ are as follows.

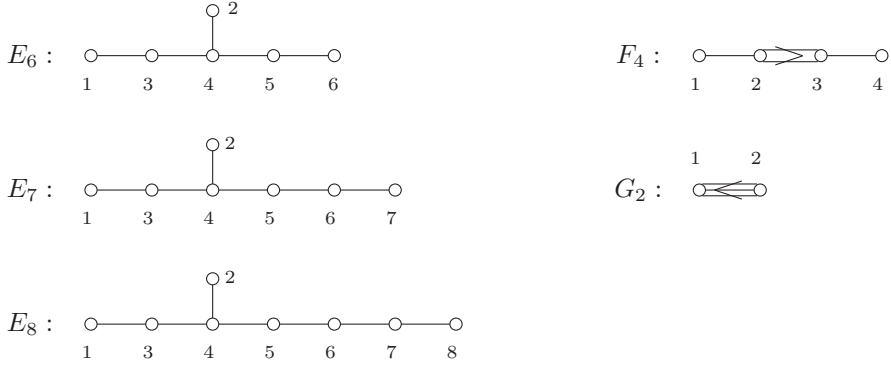
\setlength{\unitlength}{0.4mm}

\begin{center}
\begin{picture}(280,20)(0,120)
\put(10,100){\circle{4}} \put(12,100){\line(1,0){16}}
\put(30,100){\circle{4}} \put(32,100){\line(1,0){16}}
\put(50,100){\circle{4}} \put(52,100){\line(1,0){16}}
\put(70,100){\circle{4}} \put(72,100){\line(1,0){16}}
\put(90,100){\circle{4}} \put(50,102){\line(0,1){11}}
\put(50,115){\circle{4}} \put(-10,100){\makebox(0,0)[cc]{$E_6:$}}
\put(7,90){$_1$}
\put(27,90){$_3$}
\put(47,90){$_4$}
\put(67,90){$_5$}
\put(88,90){$_6$}
\put(55,115){$_2$}

\put(210,100){\circle{4}} \put(212,100){\line(1,0){16}}
\put(230,100){\circle{4}} \put(231,102){\line(1,0){18}} \put(231,98){\line(1,0){18}}
\put(250,100){\circle{4}} \put(252,100){\line(1,0){16}}
\put(270,100){\circle{4}}
 \put(190,100){\makebox(0,0)[cc]{$F_4:$}}
 \put(235,103){\line(3,-1){9}}
 \put(235,97){\line(3,1){9}}
\put(207,90){$_1$}
\put(227,90){$_2$}
\put(247,90){$_3$}
\put(267,90){$_4$}
 \end{picture}
\end{center}
\vspace{.3cm}

\begin{center}
\begin{picture}(280,20)(0,120)
\put(10,100){\circle{4}} \put(12,100){\line(1,0){16}}
\put(30,100){\circle{4}} \put(32,100){\line(1,0){16}}
\put(50,100){\circle{4}} \put(52,100){\line(1,0){16}}
\put(70,100){\circle{4}} \put(72,100){\line(1,0){16}}
\put(90,100){\circle{4}} \put(92,100){\line(1,0){16}}
\put(110,100){\circle{4}} \put(50,102){\line(0,1){11}}
\put(50,115){\circle{4}} \put(-10,100){\makebox(0,0)[cc]{$E_7:$}}
\put(7,90){$_1$}
\put(27,90){$_3$}
\put(47,90){$_4$}
\put(67,90){$_5$}
\put(87,90){$_6$}
\put(107,90){$_7$}
\put(54,115){$_2$}

\put(210,100){\circle{4}} \put(230,100){\circle{4}}
\put(211,102){\line(1,0){18}}
\put(212,100){\line(1,0){16}}
\put(211,98){\line(1,0){18}}
\put(190,100){\makebox(0,0)[cc]{$G_2:$}}
\put(207,110){$_1$}
\put(227,110){$_2$}
 \put(215,100){\line(3,-1){9}}
 \put(215,100){\line(3,1){9}}
 \end{picture}
\end{center}
\vspace{.3cm}

\begin{center}
\begin{picture}(280,20)(0,120)
\put(10,100){\circle{4}} \put(12,100){\line(1,0){16}}
\put(30,100){\circle{4}} \put(32,100){\line(1,0){16}}
\put(50,100){\circle{4}} \put(52,100){\line(1,0){16}}
\put(70,100){\circle{4}} \put(72,100){\line(1,0){16}}
\put(90,100){\circle{4}} \put(92,100){\line(1,0){16}}
\put(110,100){\circle{4}} \put(112,100){\line(1,0){16}}
\put(130,100){\circle{4}} \put(50,102){\line(0,1){11}}
\put(50,115){\circle{4}} \put(-10,100){\makebox(0,0)[cc]{$E_8:$}}
\put(7,90){$_1$}
\put(27,90){$_3$}
\put(47,90){$_4$}
\put(67,90){$_5$}
\put(87,90){$_6$}
\put(107,90){$_7$}
\put(127,90){$_8$}
\put(54,115){$_2$}
 \end{picture}
\end{center}
\vspace{2cm}

\subsection{Homogeneous vector bundles}
Now we want to introduce an important class of vector bundles on $G/P$.

\begin{define}
Over $G/P$, a vector bundle $E$ is called \emph{homogeneous} if there exists an action $G$ over $E$ such that the following diagram commutes

\centerline{
    \xymatrix{   G\times E \ar[r]\ar[d]& E\ar[d]\\
    G\times G/P \ar[r] & G/P. }}
\end{define}

\begin{rmk}
\begin{enumerate}
\item There is a one-to-one correspondence between homogeneous vector bundles on $G/P$ with holomorphic representations of $P$ (see \cite{ottaviani1995rational} Theorem 9.7).
\begin{equation}
	\begin{aligned}
		\{\text{homogeneous bundles on}~ G/P\} &\stackrel{1:1}{\longleftrightarrow}\{\text{representations of}~ P \}\\
		E&\longmapsto \upsilon:P\rightarrow GL(\pi^{-1}(P))\\
		E_{\upsilon} &\longmapsfrom \upsilon:P\rightarrow GL(V)
	\end{aligned}	
\end{equation}
\[
E_{\upsilon}:=G\times_{\upsilon}V=\{[g,v]|[g,v]\sim[gp,\upsilon(p^{-1})v] ~for ~p\in P\}.
\]
\item If a representation $\upsilon:P\to GL(V)$ is irreducible, then we call the corresponding bundle $E_{\upsilon}$ an \emph{irreducible homogeneous vector bundle}.
\end{enumerate}
\end{rmk}

Let us first recall that the \emph{weight lattice} $\Lambda$ of $G$ is the set of all linear functions $\lambda: \mathfrak{h}\to \mathbb{C}$ for which $\frac{2(\lambda,\alpha)}{(\alpha,\alpha)}\in\mathbb{Z}$ for any $\alpha\in\Phi$, where $(,)$ denotes the Killing form. An element in $\Lambda$ is called a weight. A weight $\lambda\in\Lambda$ is said to be \emph{dominant} if $\frac{2(\lambda,\alpha)}{(\alpha,\alpha)}$ are non-negative for $\alpha\in\Delta$ and \emph{strongly dominant} if these integers are positive. Let $\lambda_{1},\ldots,\lambda_{n}\in \Lambda$ be the {\it fundamental weights\/}, i.e., $\frac{2(\lambda_i,\alpha_j)}{(\alpha_j,\alpha_j)}=\delta_{ij}$. From this definition, $\lambda=\sum_{i=1}^{n}a_{i}\lambda_{i}$ is a dominant weight if and only if $a_{i}\geq 0$ and a strongly dominant weight if and only if $a_{i}>0$ for $1\leq i \leq n$.

Let $V$ be a representation of $\mathfrak{g}.$ The weight lattice of $\Lambda(V)=\{\lambda\in\Lambda| h.v=\lambda(h)v \text{ for all }h\in\mathfrak{h}\}$. A weight $\lambda\in\Lambda(V)$ is called the \emph{highest weight} of $V$ if $\lambda+\alpha$ is not a weight in $\Lambda(V)$ for any $\alpha \in \Phi^+$.

Generally, homogeneous vector bundles over $G/P$ can be classified by the filtration of the irreducible homogeneous vector bundles. Hence we only consider the irreducible homogeneous vector bundles. We now introduce the classification of the irreducible representations of parabolic subgroups.

 \begin{prop}(See \cite[Proposition 10.9]{ottaviani1995rational})
Let $I=\{\alpha_1,\dots,\alpha_k\}$ be a subset of simple roots. Let $\lambda_1,\dots,\lambda_k$ be the corresponding fundamental weights. Then all the irreducible representations of $P(I)$ are $$V\otimes L^{n_1}_{\lambda_1}\otimes\dots\otimes L^{n_k}_{\lambda_k},$$
where $V$ is a representation of $S_P$ (the semisimple part of $P(I)$), $n_i\in\mathbb{Z}$ and $L_{\lambda_i}$ is a one-dimensional representation with weight $\lambda_i$.
\end{prop}

Notice that the weight lattice of $S_P$ can be embedded in the weight lattice of $G$. If $\lambda$ is the highest weight of an irreducible representation $V$ of $S_P$, then $\lambda+\sum\limits_{u=1}^kn_u\lambda_u$ is the highest weight of $V\otimes L^{n_1}_{\lambda_1}\otimes\dots\otimes L^{n_k}_{\lambda_k}$.

\begin{rmk}\label{dominant}
\emph{\begin{enumerate}
\item[1. ]In this paper, we denote $E_\lambda$ by the homogeneous bundle arising from the irreducible representation of $P$ with highest weight $\lambda.$
\item [2.] If  $E_\lambda$ is an irreducible homogeneous vector bundle over $G/P(I)$ with $\lambda=\sum a_i\lambda_i$, then $a_i\geq0$ for $\alpha_i\notin I$. Because $E_\lambda$ is determined by the representation of the semisimple part whose
 highest weight is dominant.
\end{enumerate}}
\end{rmk}

%\begin{rmk}(\cite[Remark 10.18]{ottaviani1995rational})\label{weyl}
%A classical formula of H.Weyl expresses the dimension of the irreducible representation $G_\lambda$ of $G$ with highest weight $\lambda$. Set $\rho=\sum\limits_{i=1}^n\lambda_i$ (sum of all the fundamental weights) and $(\cdot,\cdot)$ be the Killing form. Then Weyl's formula is
%	\[
%\dim G_\lambda=\prod_{\alpha\in \Phi^+}\frac{(\lambda+\rho,\alpha)}{(\rho,\alpha)}.	
%	\]
%\end{rmk}

\subsection{ACM bundles}
We introduce ACM bundles on a projective variety as this is the main focus of this study.
\begin{define} Let $\iota:X \subset \mathbb{P}^{N}$ be a projective variety with $\mathcal{O}_{X}(1):=\iota^{*}\mathcal{O}_{\mathbb{P}^{N}}(1)$. A vector bundle, $E$, over $X$ is called {\it arithmetically\ Cohen\ Macauley} (ACM)  if
	$$H^{i}(X,E(t))=0,\ {\rm where}\ E(t):=E\otimes_{\mathcal{O}_{X}}\mathcal{O}_{X}(t),\ {\rm for\ all}\ i=1,\ldots,{\rm dim}\ X-1\ {\rm and\ all}\ t \in \mathbb{Z}.$$
\end{define}
By definition, $E$ being an ACM bundle over a projective variety, $X$, is equivalent to $E(t)$ being an ACM bundle over $X$ for $t \in \mathbb{Z}$. So for simplicity, we introduce the following definition.
\begin{define}
	Given a projective variety $(X,\mathcal{O}_X(1))$, a vector bundle $E$ on X is called \emph{initialized} if $$H^0(X,E(-1))=0$$ and $$H^0(X,E)\neq0.$$
\end{define}

For an initialized homogeneous vector bundle on a homogeneous variety, the following result is known (see \cite{du2022homogeneous}).
\begin{lem}\label{initial}
	Let $E_\lambda$ be an initialized homogeneous vector bundle on $G/P(\alpha_k)$ with $\lambda=a_1\lambda_1+\dots+a_n\lambda_n$. Then $a_{k}=0.$
\end{lem}

\section{Characterization of homogeneous ACM bundles on $G/P(\alpha_k)$}
\subsection{Proof of Theorem \ref{Thm}}
\ In this section, we prove the main result of this work by using the Borel--Bott--Weil Theorem, which is a powerful tool to compute the sheaf cohomology groups of irreducible homogeneous bundles. First, some definitions are presented.
%%%%%%%%%%%%%%%%%%
\begin{define}\ Let $\lambda$ be a weight.\\
	(i)\ $\lambda$ is called {\it singular\/} if there exists $\alpha \in \Phi^{+}$ such that $(\lambda,\alpha)=0$.\\
	(ii)\ $\lambda$ is called {\it regular\ of\ index\/} $p$ if it is not singular and if there are exactly $p$ roots $\alpha_{1},\ldots,\alpha_{p} \in \Phi^{+}$ such that $(\lambda,\alpha_{i})<0$.
\end{define}

Now we can introduce the Borel--Bott--Weil Theorem.

\begin{thm}[Borel--Bott--Weil, see \cite{ottaviani1995rational} Theorem 11.4]\label{borel bott weil} Let $E_\lambda$ be an irreducible homogeneous vector bundle over $G/P.$
	\begin{enumerate}
		\item[1)] If $\lambda+\rho$ is singular, then $$H^i(G/P,E_\lambda)=0, \forall i\in\mathbb{Z}.$$
		\item[2)] If $\lambda+\rho$ is regular of index p, then $$H^i(G/P,E_\lambda)=0, \forall i\neq p,$$
		and $$H^p(G/P,E_\lambda)=G_{w(\lambda+\rho)-\rho},$$ where $\rho=\sum\limits_{i=1}^n\lambda_i$, $w(\lambda+\rho)$ is the unique element of the fundamental Weyl chamber of G which is congruent to $\lambda+\rho$ under the action of the Weyl group and $G_{w(\lambda+\rho)-\rho}$ is the irreducible representation of $G$ with highest weight $w(\lambda+\rho)-\rho$.
	\end{enumerate}
\end{thm}

\ The following lemma is crucial to the proof of the main result.
%%%%%%%%%%%%%%%%%%
\begin{lem}\label{Lem2} Let $E_{\lambda}$ be an irreducible homogeneous vector bundle on a homogeneous variety $G/P({\alpha_{k}})$ with highest weight $\lambda$. Then, $E_{\lambda}$ is an ACM bundle if and only if one of the following conditions holds for each $t \in \mathbb{Z}$: \\
	\ \ \ $1)\ \lambda+\rho-t\lambda_{k}$ is regular  of index $0$;\\
	\ \ \ $2)\ \lambda+\rho-t\lambda_{k}$ is regular of index dim $G/P({\alpha_{k}})$; and\\
	\ \ \ $3)\ \lambda+\rho-t\lambda_{k}$ is singular.
\end{lem}
\begin{proof}
	The result follows from the definition of ACM bundles and the Borel--Bott--Weil Theorem.
\end{proof}

\begin{proof}[Proof of Theorem \ref{Thm}] We examine these conditions by using Lemma \ref{Lem2}.
	%%%%%%%%%%%%%%%%%%
\begin{lem}\ $\lambda+\rho-t\lambda_{k}$ is regular of index $0$ if and only if $t<1$.
	\end{lem}
	\begin{proof}
		Suppose that $\lambda+\rho-t\lambda_{k}$ is regular of index $0$. Then, $\lambda+\rho-t\lambda_{k}=\sum_{i\neq k}(1+a_{i})\lambda_{i}+(1-t+a_{k})\lambda_{k}$ is strongly dominant, i.e., $a_{i}+1>0~(i \neq k)$ and $1+a_{k}-t>0$. Now, as $a_{i}\geq0$ for $i \neq k$ (Remark \ref{dominant}), we only consider the case $1+a_{k}-t>0$. By initializing (Lemma \ref{initial}), we obtain $a_{k}=0$. Therefore, $t<1$. Conversely, suppose $t<1$. Using the same argument, $\lambda+\rho-t\lambda_{k}$ is regular of index $0$. Hence, $\lambda+\rho-t\lambda_{k}$ is regular of index $0$ if and only if $t<1$.
	\end{proof}
	%%%%%%%%%%%%%%%%%
	Next, we consider case $(2)$.
	%%%%%%%%%%%%%%%%%%
	\begin{lem}\ $\lambda+\rho-t\lambda_{k}$ is regular of index $\dim G/P({\alpha_{k}})$ if and only if $t>M^{G}_{\lambda,k}$.
	\end{lem}
	\begin{proof}
Since $\dim G/P({\alpha_{k}})=|\Phi_{k,G}^{+}|$, it's equivalent to prove that $\lambda+\rho-t\lambda_{k}$ is regular of index $|\Phi_{k,G}^{+}|$ if and only if $t>M^{G}_{\lambda,k}$. Note that $(\lambda+\rho-t\lambda_{k},\alpha)=(\lambda+\rho,\alpha)-(\lambda_{k},\alpha)\cdot t$. Thus, for any $\alpha\in \Phi^{+}_{G}\setminus\Phi_{k,G}^{+}$, $(\lambda+\rho-t\lambda_{k},\alpha)=(\lambda+\rho,\alpha)>0$, because $a_{i}\geq0~(i \neq k)$ and $(\rho,\alpha)>0$.

Suppose that $t>M^{G}_{\lambda,k}$. By the definition of $M^{G}_{\lambda,k}$, the pairing value $(\lambda+\rho-t\lambda_{k},\alpha)$ is negative for any $\alpha \in \Phi_{k,G}^{+}$. Therefore, $\lambda+\rho-t\lambda_{k}$ is regular of index $|\Phi_{k,G}^{+}|$. Conversely, if $\lambda+\rho-t\lambda_{k}$ is regular of index $|\Phi_{k,G}^{+}|$, then the pairing value $(\lambda+\rho-t\lambda_{k},\alpha)$ is negative for all $\alpha \in \Phi_{k,G}^{+}$. In particular, $t>M^{G}_{\lambda,k}$.
	\end{proof}
	%%%%%%%%%%%%%%%%%%%
	Therefore, $E_{\lambda}$ being an ACM bundle is equivalent to $\lambda+\rho-t\lambda_{k}$ being singular for all integers $t \in [1,M^{G}_{\lambda,k}]$ (i.e., there exists a positive root $\alpha$ such that $(\lambda+\rho-t\lambda_{k},\alpha)=0$ for any integer $t \in [1,M^{G}_{\lambda,k}]$). If $\alpha$ is in $\Phi^{+}_{G}\setminus\Phi_{k,G}^{+}$, then the pairing value $(\lambda+\rho-t\lambda_{k},\alpha)=(\lambda+\rho,\alpha)$. As $a_{i}\geq0~(i \neq k)$ and $(\rho,\alpha)>0$, this is positive. Hence, we consider only the case in which $\alpha$ is in $\Phi_{k,G}^{+}$.
	%%%%%%%%%%%%%%%%%%%
	\begin{lem}\ $\lambda+\rho-t\lambda_{k}$ is singular if and only if $t\in T_{\lambda,k}^{G}$.
	\end{lem}
	\begin{proof} If $\lambda+\rho-t\lambda_{k}$ is singular, there exists a positive root $\alpha$ in $\Phi_{k,G}^{+}$ such that
		$$0=(\lambda+\rho-t\lambda_{k},\alpha)=(\lambda+\rho,\alpha)-(\lambda_{k},\alpha)\cdot t.$$
		As $\alpha$ is in $\Phi_{k,G}^{+}$, $(\lambda_{k},\alpha)$ is not equal to zero. Thus, $t=\frac{1}{(\lambda_{k},\alpha)}(\lambda+\rho,\alpha)$. By the definition of $T_{\lambda,k}^{G}$, $t\in T_{\lambda,k}^{G}$. Conversely, if $t$ is in $T_{\lambda,k}^{G}$, there exists a positive root $\alpha$ in $\Phi_{k,G}^{+}$ such that $t=\frac{1}{(\lambda_{k},\alpha)}(\lambda+\rho,\alpha)$. Then,
		$$0=(\lambda+\rho,\alpha)-(\lambda_{k},\alpha)\cdot t=(\lambda+\rho-t\lambda_{k},\alpha).$$
		Therefore, $\lambda+\rho-t\lambda_{k}$ is singular.
	\end{proof}
	Therefore, in summary, $E_{\lambda}$ is an ACM vector bundle if and only if $n_l\geq1$ for any integer $l\in [1,M^G_{k,\lambda}]$.
\end{proof}

Now we can prove Corollary \ref{Cor} which is directly induced by our main theorem.

\begin{proof}[Proof of Corollary \ref{Cor}]Let $E_{\lambda}$ be an irreducible homogeneous vector bundle with $\lambda=\sum_{i=1}^{n}a_{i}\lambda_{i}$. We assume that $a_{k}=0$, without loss of generality. As $\dim G/P({\alpha_{k}})=|\Phi_{k,G}^{+}|$, $T^{G}_{\lambda,k}$ has at most $\dim G/P({\alpha_{k}})$ different elements. Thus if $M^{G}_{\lambda,k}>{\dim}\ G/P(\alpha_{k})$, there exists an integer $l \in[1,M^{G}_{\lambda,k}]$ such that $n_{l}=0$. By the main theorem, $E_{\lambda}$ is not an ACM bundle. Therefore, if $E_{\lambda}$ is an ACM bundle, then $M^{G}_{\lambda,k}\leq {\rm dim}\ G/P(\alpha_{k})$. As $M^{G}_{\lambda,k}$ is a linear combination of $a_{i}~(i \neq k)$ with positive coefficients and $a_{i}\geq0~(i \neq k)$, there exist only a finite number of choices for $a_{i}$. Therefore, there exist only a finite number of irreducible homogeneous ACM bundles.
\end{proof}

%From now on, we denote $n_{l}:=\#\{\alpha\in \Phi_{k,G}^{+} |\ \frac{1}{(\lambda_{k},\alpha)}(\lambda+\rho,\alpha)=l\}$. Then the main theorem can be rephrased as follows. Let $E_{\lambda}$ be an initialized irreducible homogeneous vector bundle over $G/P(\alpha_{k})$ with highest  weight $\lambda$. Then, $E_{\lambda}$ is an ACM vector bundle if and only if $n_{l}\geq1$ for any integer $l\in[1,M^{G}_{\lambda,k}]$.

\subsection{ Homogeneous ACM bundles on exceptional Grassmannians}
In this subsection, we give a further observation of $T_{\lambda,k}^{G}$ on homogeneous varieties of different types. When $G$ is a simple Lie group of type $A$, $B$, $C$ or $D$, the concrete form of $T_{\lambda,k}^{G}$ has already been given in \cite{costa2016homogeneous} and \cite{du2022homogeneous} (see the definition of step matrix). Therefore, we only need to consider the concrete form of $T_{\lambda,k}^{G}$ when $G$ is a simple Lie group of type $E_n(n=6,7,8)$, $F_4$ or $G_2$.

 To give a more concrete description, we need to know the explicit form of positive roots.
In \cite{bourbaki2005elements}{ Plate V-IX,} we have already know the form of the positive roots under orthonormal basis.
\begin{lem}

    \[\Phi_{E_6}^+=\{\pm e_i-e_j~(1\leq i<j\leq 5),\frac{1}{2}(e_8-e_7-e_6+\sum_{k=1}^5(-1)^{v(k)}e_k)\};\]
    \[\Phi_{E_7}^+=\Phi_{E_6}^+\cup\{\pm e_i+e_6~(i<6),e_8-e_7,\frac{1}{2}(e_8-e_7+e_6+\sum_{k=1}^5(-1)^{v(k)}e_k)\};\]
    \[\Phi_{E_8}^+=\Phi_{E_7}^+\cup\{\pm e_i+e_7,\pm e_i+e_8~(i<7),e_8+e_7,\frac{1}{2}(e_8+e_7+\sum_{k=1}^6(-1)^{u(k)}e_k)\},\]
    where $\sum_{k=1}^5v(k)$ is even and $\sum_{k=1}^6u(k)$ is odd.
    \[\Phi_{F_4}^+=\{e_i~(1\leq i\leq 4),e_i\pm e_j(1\leq i<j\leq 4),\frac{1}{2}(e_1\pm e_2\pm e_3\pm e_4)\}.\]
  \[\Phi^+_{G_2}=\{e_1-e_2,-2e_1+e_2+e_3,-e_1+e_3,-e_2+e_3,e_1-2e_2+e_3,-e_1-e_2+2e_3
    \}.\]
\end{lem}

For convenience, we express positive roots as  combinations of simple roots. After some simple calculations, we can get the following lemma (we use $(m_1,\dots,m_n)$ to represent $ m_1\alpha_1+\dots+m_n\alpha_n$ and list them in terms of the lexicographical order).

  \begin{lem}\label{posroot}

  \begin{enumerate}
      \item Let $\alpha_1,\dots,\alpha_8$ be the simple roots of $E_8$ and $\Phi_P^+,$ $\Phi_Q^+$ and $\Phi_R^+$ be the subset of the positive roots of $E_8$, where

$\Phi_P^+=\{$
\begin{footnotesize}
\[\begin{matrix}
(	0	,	0	,	0	,	0	,	0	,	1	,	0	,	0	)	,&
(	0	,	0	,	0	,	0	,	1	,	0	,	0	,	0	)	,&
(	0	,	0	,	0	,	0	,	1	,	1	,	0	,	0	)	,&
(	0	,	0	,	0	,	1	,	0	,	0	,	0	,	0	)	,&
(	0	,	0	,	0	,	1	,	1	,	0	,	0	,	0	)	,
\\
(	0	,	0	,	0	,	1	,	1	,	1	,	0	,	0	)	,&
(	0	,	0	,	1	,	0	,	0	,	0	,	0	,	0	)	,&
(	0	,	0	,	1	,	1	,	0	,	0	,	0	,	0	)	,&
(	0	,	0	,	1	,	1	,	1	,	0	,	0	,	0	)	,&
(	0	,	0	,	1	,	1	,	1	,	1	,	0	,	0	)	,
\\
(	0	,	1	,	0	,	0	,	0	,	0	,	0	,	0	)	,&
(	0	,	1	,	0	,	1	,	0	,	0	,	0	,	0	)	,&
(	0	,	1	,	0	,	1	,	1	,	0	,	0	,	0	)	,&
(	0	,	1	,	0	,	1	,	1	,	1	,	0	,	0	)	,&
(	0	,	1	,	1	,	1	,	0	,	0	,	0	,	0	)	,
\\
(	0	,	1	,	1	,	1	,	1	,	0	,	0	,	0	)	,&
(	0	,	1	,	1	,	1	,	1	,	1	,	0	,	0	)	,&
(	0	,	1	,	1	,	2	,	1	,	0	,	0	,	0	)	,&
(	0	,	1	,	1	,	2	,	1	,	1	,	0	,	0	)	,&
(	0	,	1	,	1	,	2	,	2	,	1	,	0	,	0	)	,
\\
(	1	,	0	,	0	,	0	,	0	,	0	,	0	,	0	)	,&
(	1	,	0	,	1	,	0	,	0	,	0	,	0	,	0	)	,&
(	1	,	0	,	1	,	1	,	0	,	0	,	0	,	0	)	,&
(	1	,	0	,	1	,	1	,	1	,	0	,	0	,	0	)	,&
(	1	,	0	,	1	,	1	,	1	,	1	,	0	,	0	)	,
\\
(	1	,	1	,	1	,	1	,	0	,	0	,	0	,	0	)	,&
(	1	,	1	,	1	,	1	,	1	,	0	,	0	,	0	)	,&
(	1	,	1	,	1	,	1	,	1	,	1	,	0	,	0	)	,&
(	1	,	1	,	1	,	2	,	1	,	0	,	0	,	0	)	,&
(	1	,	1	,	1	,	2	,	1	,	1	,	0	,	0	)	,
\\
(	1	,	1	,	1	,	2	,	2	,	1	,	0	,	0	)	,&
(	1	,	1	,	2	,	2	,	1	,	0	,	0	,	0	)	,&
(	1	,	1	,	2	,	2	,	1	,	1	,	0	,	0	)	,&
(	1	,	1	,	2	,	2	,	2	,	1	,	0	,	0	)	,&
(	1	,	1	,	2	,	3	,	2	,	1	,	0	,	0	)	,
\\
(	1	,	2	,	2	,	3	,	2	,	1	,	0	,	0	)\};& &&&
\end{matrix}\]
\end{footnotesize}

$\Phi_Q^+=\{$
\begin{footnotesize}
\[\begin{matrix}
(	0	,	0	,	0	,	0	,	0	,	0	,	1	,	0	)	,&	
(	0	,	0	,	0	,	0	,	0	,	1	,	1	,	0	)	,&	
(	0	,	0	,	0	,	0	,	1	,	1	,	1	,	0	)	,&	
(	0	,	0	,	0	,	1	,	1	,	1	,	1	,	0	)	,&	
(	0	,	0	,	1	,	1	,	1	,	1	,	1	,	0	)	,&	\\
(	0	,	1	,	0	,	1	,	1	,	1	,	1	,	0	)	,&	
(	0	,	1	,	1	,	1	,	1	,	1	,	1	,	0	)	,&	
(	0	,	1	,	1	,	2	,	1	,	1	,	1	,	0	)	,&	
(	0	,	1	,	1	,	2	,	2	,	1	,	1	,	0	)	,&	
(	0	,	1	,	1	,	2	,	2	,	2	,	1	,	0	)	,&	\\
(	1	,	0	,	1	,	1	,	1	,	1	,	1	,	0	)	,&	
(	1	,	1	,	1	,	1	,	1	,	1	,	1	,	0	)	,&	
(	1	,	1	,	1	,	2	,	1	,	1	,	1	,	0	)	,&	
(	1	,	1	,	1	,	2	,	2	,	1	,	1	,	0	)	,&	
(	1	,	1	,	1	,	2	,	2	,	2	,	1	,	0	)	,&	\\
(	1	,	1	,	2	,	2	,	1	,	1	,	1	,	0	)	,&	
(	1	,	1	,	2	,	2	,	2	,	1	,	1	,	0	)	,&	
(	1	,	1	,	2	,	2	,	2	,	2	,	1	,	0	)	,&	
(	1	,	1	,	2	,	3	,	2	,	1	,	1	,	0	)	,&	
(	1	,	1	,	2	,	3	,	2	,	2	,	1	,	0	)	,&	\\
(	1	,	1	,	2	,	3	,	3	,	2	,	1	,	0	)	,&	
(	1	,	2	,	2	,	3	,	2	,	1	,	1	,	0	)	,&	
(	1	,	2	,	2	,	3	,	2	,	2	,	1	,	0	)	,&	
(	1	,	2	,	2	,	3	,	3	,	2	,	1	,	0	)	,&	
(	1	,	2	,	2	,	4	,	3	,	2	,	1	,	0	)	,&	\\
(	1	,	2	,	3	,	4	,	3	,	2	,	1	,	0	)	,&	
(	2	,	2	,	3	,	4	,	3	,	2	,	1	,	0	)	\};&	

\end{matrix}\]
\end{footnotesize}

$\Phi_R^+=\{$
\begin{footnotesize}
\[\begin{matrix}
(	0	,	0	,	0	,	0	,	0	,	0	,	0	,	1	)	,&	
(	0	,	0	,	0	,	0	,	0	,	0	,	1	,	1	)	,&	
(	0	,	0	,	0	,	0	,	0	,	1	,	1	,	1	)	,&	
(	0	,	0	,	0	,	0	,	1	,	1	,	1	,	1	)	,&	
(	0	,	0	,	0	,	1	,	1	,	1	,	1	,	1	)	,&	\\
(	0	,	0	,	1	,	1	,	1	,	1	,	1	,	1	)	,&	
(	0	,	1	,	0	,	1	,	1	,	1	,	1	,	1	)	,&	
(	0	,	1	,	1	,	1	,	1	,	1	,	1	,	1	)	,&	
(	0	,	1	,	1	,	2	,	1	,	1	,	1	,	1	)	,&	
(	0	,	1	,	1	,	2	,	2	,	1	,	1	,	1	)	,&	\\
(	0	,	1	,	1	,	2	,	2	,	2	,	1	,	1	)	,&	
(	0	,	1	,	1	,	2	,	2	,	2	,	2	,	1	)	,&	
(	1	,	0	,	1	,	1	,	1	,	1	,	1	,	1	)	,&	
(	1	,	1	,	1	,	1	,	1	,	1	,	1	,	1	)	,&	
(	1	,	1	,	1	,	2	,	1	,	1	,	1	,	1	)	,&	\\
(	1	,	1	,	1	,	2	,	2	,	1	,	1	,	1	)	,&	
(	1	,	1	,	1	,	2	,	2	,	2	,	1	,	1	)	,&	
(	1	,	1	,	1	,	2	,	2	,	2	,	2	,	1	)	,&	
(	1	,	1	,	2	,	2	,	1	,	1	,	1	,	1	)	,&	
(	1	,	1	,	2	,	2	,	2	,	1	,	1	,	1	)	,&	\\
(	1	,	1	,	2	,	2	,	2	,	2	,	1	,	1	)	,&	
(	1	,	1	,	2	,	2	,	2	,	2	,	2	,	1	)	,&	
(	1	,	1	,	2	,	3	,	2	,	1	,	1	,	1	)	,&	
(	1	,	1	,	2	,	3	,	2	,	2	,	1	,	1	)	,&	
(	1	,	1	,	2	,	3	,	2	,	2	,	2	,	1	)	,&	\\
(	1	,	1	,	2	,	3	,	3	,	2	,	1	,	1	)	,&	
(	1	,	1	,	2	,	3	,	3	,	2	,	2	,	1	)	,&	
(	1	,	1	,	2	,	3	,	3	,	3	,	2	,	1	)	,&	
(	1	,	2	,	2	,	3	,	2	,	1	,	1	,	1	)	,&	
(	1	,	2	,	2	,	3	,	2	,	2	,	1	,	1	)	,&	\\
(	1	,	2	,	2	,	3	,	2	,	2	,	2	,	1	)	,&	
(	1	,	2	,	2	,	3	,	3	,	2	,	1	,	1	)	,&	
(	1	,	2	,	2	,	3	,	3	,	2	,	2	,	1	)	,&	
(	1	,	2	,	2	,	3	,	3	,	3	,	2	,	1	)	,&	
(	1	,	2	,	2	,	4	,	3	,	2	,	1	,	1	)	,&	\\
(	1	,	2	,	2	,	4	,	3	,	2	,	2	,	1	)	,&	
(	1	,	2	,	2	,	4	,	3	,	3	,	2	,	1	)	,&	
(	1	,	2	,	2	,	4	,	4	,	3	,	2	,	1	)	,&	
(	1	,	2	,	3	,	4	,	3	,	2	,	1	,	1	)	,&	
(	1	,	2	,	3	,	4	,	3	,	2	,	2	,	1	)	,&	\\
(	1	,	2	,	3	,	4	,	3	,	3	,	2	,	1	)	,&	
(	1	,	2	,	3	,	4	,	4	,	3	,	2	,	1	)	,&	
(	1	,	2	,	3	,	5	,	4	,	3	,	2	,	1	)	,&	
(	1	,	3	,	3	,	5	,	4	,	3	,	2	,	1	)	,&	
(	2	,	2	,	3	,	4	,	3	,	2	,	1	,	1	)	,&	\\
(	2	,	2	,	3	,	4	,	3	,	2	,	2	,	1	)	,&	
(	2	,	2	,	3	,	4	,	3	,	3	,	2	,	1	)	,&	
(	2	,	2	,	3	,	4	,	4	,	3	,	2	,	1	)	,&	
(	2	,	2	,	3	,	5	,	4	,	3	,	2	,	1	)	,&	
(	2	,	2	,	4	,	5	,	4	,	3	,	2	,	1	)	,&	\\
(	2	,	3	,	3	,	5	,	4	,	3	,	2	,	1	)	,&	
(	2	,	3	,	4	,	5	,	4	,	3	,	2	,	1	)	,&	
(	2	,	3	,	4	,	6	,	4	,	3	,	2	,	1	)	,&	
(	2	,	3	,	4	,	6	,	5	,	3	,	2	,	1	)	,&	
(	2	,	3	,	4	,	6	,	5	,	4	,	2	,	1	)	,&	\\
(	2	,	3	,	4	,	6	,	5	,	4	,	3	,	1	)	,&	
(	2	,	3	,	4	,	6	,	5	,	4	,	3	,	2	)	\}.&	

\end{matrix}\]
\end{footnotesize}
Then $\Phi_{E_6}^+=\Phi_P^+$, $\Phi_{E_7}^+=\Phi_P^+\cup \Phi_Q^+$ and $\Phi_{E_8}^+=\Phi_P^+\cup \Phi_Q^+\cup \Phi_R^+$.

      \item Let $\alpha_1,\dots,\alpha_4$ be the simple roots of $F_4$. Then

 $\Phi_{F_4}^+=\{$
\[\begin{matrix}
(	0	,	0	,	0	,	1	)	,	&	
(	0	,	0	,	1	,	0	)	,	&	
(	0	,	0	,	1	,	1	)	,	&	
(	0	,	1	,	0	,	0	)	,	&	
(	0	,	1	,	1	,	0	)	,	&	
(	0	,	1	,	1	,	1	)	,		\\
(	0	,	1	,	2	,	0	)	,	&	
(	0	,	1	,	2	,	1	)	,	&	
(	0	,	1	,	2	,	2	)	,	&	
(	1	,	0	,	0	,	0	)	,	&	
(	1	,	1	,	0	,	0	)	,	&	
(	1	,	1	,	1	,	0	)	,		\\
(	1	,	1	,	1	,	1	)	,	&	
(	1	,	1	,	2	,	0	)	,	&	
(	1	,	1	,	2	,	1	)	,	&	
(	1	,	1	,	2	,	2	)	,	&	
(	1	,	2	,	2	,	0	)	,	&	
(	1	,	2	,	2	,	1	)	,		\\
(	1	,	2	,	2	,	2	)	,	&	
(	1	,	2	,	3	,	1	)	,	&	
(	1	,	2	,	3	,	2	)	,	&	
(	1	,	2	,	4	,	2	)	,	&	
(	1	,	3	,	4	,	2	)	,	&	
(	2	,	3	,	4	,	2	)	\};	
\end{matrix}\]

    \item Let $\alpha_1$ and $\alpha_2$ be the simple roots of $G_2$. Then \[\Phi^+_{G_2}=\{(0,1),(1,0),(1,1),(2,1),(3,1),(3,2)
    \}.\]

  \end{enumerate}

\end{lem}
\begin{rmk}
 Here we embed the roots of $E_6$ and $E_7$ into the roots of $E_8$. Hence we only use the roots of $E_8$ to represent others.
\end{rmk}

From the definition of the fundamental weight, it is not hard to see that \[\Phi_{k,G}^{+}=\{\alpha=\sum m_i\alpha_i\in\Phi^+_{G} |m_k\neq 0\},\]
 \[
 (\lambda_{k},\alpha)= (\lambda_{k},\sum m_i\alpha_i)=\frac{1}{2}m_k(\alpha_k,\alpha_k)
 \]
 and for $\lambda=a_1\lambda_1+\dots+a_n\lambda_n$,
\[
(\lambda+\rho,\alpha)=\frac{1}{2}\sum (a_i+1) m_i(\alpha_i,\alpha_i).
\]

Combining Lemma \ref{posroot} and the Killing forms of positive roots, we can write down the concrete form of $T^{G}_{\lambda,k}$ as follows.

 \begin{rmk}\label{T}
Let $E_\lambda$ be an irreducible homogeneous vector bundle over $G/P(\alpha_k)$ with $\lambda=a_1\lambda_1+\dots+a_n\lambda_n$. Then
\begin{itemize}
	\item \[
	T^{E_n}_{\lambda,k}=\big\{\frac{\sum_{i}(a_i+1)m_i}{m_k} |(m_1,\ldots,m_n)\in\Phi^+_{E_n},~m_k\neq 0\big\}~(n=6,7,8).
	\]
	
	\item If $k=1,2$, then
\begin{small}

\[
T^{F_4}_{\lambda,k}=\big\{\frac{2(a_1+1)m_1+2(a_2+1)m_2+(a_3+1)m_3+(a_4+1)m_4}{2m_k}\ |(m_1,\ldots,m_4)\in\Phi^+_{F_4} ,~m_k\neq 0\big\}.	
	\]\end{small}	
If $k=3,4$, then
\begin{small}
\[
T^{F_4}_{\lambda,k}=\big\{\frac{2(a_1+1)m_1+2(a_2+1)m_2+(a_3+1)m_3+(a_4+1)m_4}{ m_k}\ |(m_1,\ldots,m_4)\in\Phi^+_{F_4} ,~m_k\neq 0\big\}.
\]
\end{small}
	\item \[
T^{G_2}_{\lambda,1}=\big \{\frac{ (a_1+1)m_1+ 3(a_2+1)m_2}{m_1}\ |(m_1,m_2)\in\Phi^+_{G_2},~m_1\neq 0	\big\}
	\]
	and
\[
T^{G_2}_{\lambda,2}=\big\{                  \frac{(a_1+1)m_1+ 3(a_2+1)m_2}{3m_2}\ |(m_1,m_2)\in\Phi^+_{G_2},~m_2\neq 0	\big\}.
\]		
\end{itemize}
\end{rmk}

Let's illustrate the power of this result by means of some examples.
\begin{ex}\ Let $E_{\lambda}$ and $E_{\mu}$ be initialized irreducible homogeneous vector bundles on $E_{6}/P(\alpha_{2})$ with $\mu=2\lambda_{1}+\lambda_{3}$ and $\nu=\lambda_{4}+\lambda_{5}$. Then by Lemma \ref{posroot} and Remark \ref{T}, we know that for $\lambda=a_1\lambda_1+\dots+a_6\lambda_6$,
\begin{align*}
T_{\lambda,2}^{E_{6}}&=\big\{a_2+1,a_2+a_4+2,a_2+\sum_{i=4}^{5}a_i+3,a_2+\sum_{i=4}^{6}a_i+4,\sum_{i=2}^{4}a_i+3,\sum_{i=2}^{5}a_i+4,\sum_{i=2}^{6}a_i+5,\\
&\sum_{i=2}^{3}a_i+2a_4+a_5+5,\sum_{i=2}^{3}a_i+2a_4+\sum_{i=5}^{6}a_i+6,\sum_{i=2}^{3}a_i+2\sum_{i=4}^{5}a_i+a_6+7,\sum_{i=1}^{4}a_i+4,\\
&\sum_{i=1}^{5}a_i+5,\sum_{i=1}^{6}a_i+6,\sum_{i=1}^{3}a_i+2a_4+a_5+6,\sum_{i=1}^{3}a_i+2a_4+\sum_{i=5}^{6}a_i+7,\sum_{i=1}^{3}a_i+2\sum_{i=4}^{5}a_i+a_6+8,\\
&\sum_{i=1}^{2}a_i+2\sum_{i=3}^{4}a_i+a_5+7,\sum_{i=1}^{2}a_i+2\sum_{i=3}^{4}a_i+\sum_{i=5}^{6}a_i+8,\sum_{i=1}^{2}a_i+2\sum_{i=3}^{5}a_i+a_6+9,\\
&\sum_{i=1}^{2}a_i+2a_3+3a_4+2a_5+a_6+10,\frac{1}{2}(a_1+2\sum_{i=2}^{3}a_i+3a_4+2a_5+a_6+11)
\big\}.
\end{align*}
Hence
\[
T_{\mu,2}^{E_{6}}=\{1,2,3,4,4,5,6,
6,7,8,7,
8,9,9,10,11,
11,12,13,
14,\frac{15}{2}\}~\text{and}~ M_{\mu,2}^{E_{6}}=14,\]

\[
T_{\nu,2}^{E_{6}}=\{
1,3,5,6,4,6,7,
8,9,11,5,
7,8,9,10,12,
10,11,13,
15,8\}~\text{and}~ M_{\nu,2}^{E_{6}}=15.\]

It's easy to verify that $n_{l}\geq1$ for any integer $l\in[1,M_{\mu,2}^{E_{6}}]$ and $n_{2}=n_{14}=0$ for $T_{\nu,2}^{E_{6}}$.
Therefore, by Theorem \ref{Thm}, $E_{\mu}$ is an ACM bundle, but $E_{\nu}$ is not an ACM bundle.
\end{ex}

%Since the number of the elements in associated datum of type $G_2$ is not large. It is not hard to give a complete classification of irreducible homogeneous ACM bundles on $G_2/P(\alpha_k)$.

\begin{ex}\label{classG} \begin{enumerate}
       \item  The irreducible homogeneous ACM vector bundles over $G_2/P(\alpha_1)$ are line bundles.

     \item  The irreducible homogeneous ACM vector bundles over $G_2/P(\alpha_2)$ are line bundles, $E_{\lambda_1}$ and $E_{2\lambda_1}$ (up to tensoring a line bundle).

 \end{enumerate}

 \end{ex}
  \begin{proof}

  \begin{enumerate}
 \item For simplicity, we may assume that $E_\lambda=E_{a_1\lambda_1+a_2\lambda_2}$ be an initialized irreducible homogeneous ACM bundle on $G_2/P(\alpha_1).$ Then $a_1=0.$ By Lemma \ref{posroot} and Remark \ref{T}, we have \[T_{1,\lambda}^G=\{ 1,3a_2 + 4, \frac{3a_2}{2} + \frac{5}{2} ,a_2 +2,2a_2+ 3\}\]
  By Theorem \ref{Thm}, $2\in T_{2,\lambda}^G$, then the only choice of $a_2$ is 0. Then
  \[T_{1,\lambda}^G=\{ 1, 4, \frac{5}{2} , 2,3\}.\]
  Hence the first statement follows from Theorem \ref{Thm}.
      \item As above, we still assume that $E_\lambda=E_{a_1\lambda_1+a_2\lambda_2}$ is initialized which means $a_2=0$. Then we have \[T_{2,\lambda}^G=\{ 1,\frac{a_1}{3} +\frac{4}{3} ,\frac{2a_1}{3}  + \frac{5}{3} ,a_1 +2,\frac{a_1}{2} + \frac{3}{2}\}\]
By Theorem \ref{Thm}, $2\in T_{2,\lambda}^G$, then the only choices of $a_1$ are 0,1 and 2. One can check
  \[T_{2,\lambda}^G=\{ 1,\frac{4}{3}  ,\frac{5}{3}  , 2,\frac{3}{2} \},\{ 1,\frac{5}{3},\frac{7}{3} ,3,2\}~\text{and}~\{ 1,2, 3 ,4,\frac{5}{2}\}\]
 % \[T_{2,\lambda}^G=\{ 1,\frac{5}{3},\frac{7}{3} ,3,2\};\] \[T_{2,\lambda}^G=\{ 1,2, 3 ,4,\frac{5}{2}\}.\]
 in these three cases. We therefore conclude by Theorem \ref{Thm}.
  \end{enumerate}
   \end{proof}
  In general, it is not hard to use computer to help us classify all irreducible homogeneous ACM bundles on exceptional Grassmannians. We list part of them in the appendix. For convenience, we only give the following corollary which we will use in the next section.

\begin{cor}\label{Class2}
Let $X$ be an exceptional Grassmannian and $F_1$, $F_2$ be two irreducible homogeneous vector bundles with the forms in Table \ref{table1}. Then $F_1$ and $F_2$ are ACM bundles.

  \begin{center}
               \begin{table}[h]
                   \centering
\begin{tabular}{|c|c|c| }
  \hline
$X$ &$F_1$ &$F_2$ \\
\hline
$E_7/P(\alpha_1)$ & $E_{2\lambda_2}$  &  $E_{2\lambda_7-2\lambda_1}$   \\\hline
$E_n/P(\alpha_2)$ & $E_{3\lambda_1}$ & $E_{\lambda_3-2\lambda_2}$    \\\hline
$E_n/P(\alpha_3)$ & $E_{2\lambda_1}$& $E_{\lambda_4-2\lambda_3}$    \\\hline
$E_n/P(\alpha_4)$ & $ E_{\lambda_1+\lambda_2+\lambda_3}$ &$E_{\lambda_5-2\lambda_4}$    \\\hline
$E_n/P(\alpha_5)$ & $E_{ \lambda_1+\lambda_4}$ & $E_{\lambda_6-2\lambda_5}$     \\\hline
$E_n/P(\alpha_6)$~$(n\neq 6)$ &  $ E_{\lambda_1+\lambda_5}$ & $E_{\lambda_7-2\lambda_6}$    \\\hline
$E_8/P(\alpha_7)$ & $ E_{\lambda_1+\lambda_6}$ & $E_{\lambda_8-2\lambda_7}$   \\\hline
$F_4/P(\alpha_1)$ & $ E_{3\lambda_4}$ & $ E_{\lambda_4-2\lambda_1}$  \\\hline
$F_4/P(\alpha_2)$ &  $ E_{2\lambda_1}$ & $ E_{2\lambda_3-2\lambda_2}$   \\\hline
$F_4/P(\alpha_3)$ & $ E_{\lambda_1+\lambda_2}$ & $ E_{\lambda_4-2\lambda_3}$  \\\hline
\end{tabular}
                   \caption{Choices of $F_1$ and $F_2$}
                   \label{table1}
               \end{table}
\end{center}
\end{cor}

 \section{Representation type of exceptional Grassmannians}

 %As an application, we can show that the representation type of some exceptional isotropic Grassmannians are of wild representation type.
In this section, we present certain applications for the main theorem. First, let's recall that a subscheme $X\subset\mathbb{P}^N$ is ACM if its homogeneous coordinate ring $R_X$ is a local Cohen-Macauley ring (See \cite{migliore1998introduction} Definition 1.2.2). We also note that exceptional Grassmannians are ACM schemes (\cite{ramanathan1985schubert}). Mimicking an analogus trichotomy in Representation Theory, a classification of ACM schemes was proposed as finite, tame or wild according to the complexity of their associated category of ACM bundles. Let us recall the definitions:

 \begin{define}
 Let $X\subset\mathbb{P}^n$ be an nonsingular ACM subscheme of dimension $d$. $X$ is said to be of \emph{finite representation type} if it has, up to twist and isomorphism, only a finite number of indecomposable ACM bundles. It is said to be of \emph{tame representation type} if either it has, up to twist and isomorphism, an infinite discrete set of indecomposable ACM bundles or, for each rank $r$, the indecomposable ACM bundles of rank $r$ form a finite number of families of dimension at most $d.$ Finally, $X$ is said to be of \emph{wild representation type} if there exist $l$-dimensional families of non-isomorphic indecomposable ACM bundles for arbitrary large $l.$
 \end{define}

Varieties of finite representation type have been completely classified and they fall into a short list: three or less reduced points on $\mathbb{P}^2$, a rational normal curve, a projective space, a non-singular quadric hypersurface, a cubic scroll in $\mathbb{P}^4$ and the Veronese surface in $\mathbb{P}^5$ (\cite{buchweitz1987cohen}, \cite{eisenbud1988classification}).
Recently, Joan Pons-Llopis and Daniele Faenzi showed that most ACM varieties are of wild representation type (\cite{pons2021cohen}). However, we can also use our classification theorem to construct $l$-dimensional families of non-isomorphic indecomposable ACM bundles for arbitrary large $l$ on some exceptional Grassmannians, which means they are of wild representation type.

In this section, we suppose that $X$ is one of the following exceptional Grassmannians: \[E_n/P(\alpha_k)~,F_4/(P(\alpha_l)),\]
 where $(n,k)\neq (6,1),(7,7),(8,1),(8,8)$ and $l\neq 4$ (i.e., $X$ is the variety listed in Table \ref{table1}).

 We need the following lemma which can be found in \cite{snow1989homogeneous}.
 \begin{lem}\label{canonical}
 The canonical bundle on $X=G/P(\alpha_k)$
 is $\mathcal{O}_X(m)$. Here we list their relations as follows.
 \begin{enumerate}
     \item  For $E_6$,
\begin{center}
         \begin{tabular}{|c|c|c|c|c|c|c|}
  \hline
k & 1&2 &3&4&5&6 \\
\hline
m & -12&-11 &-9&-7&-9&-12  \\\hline
\end{tabular}
\end{center}

     \item  For $E_7$,

\begin{center}
         \begin{tabular}{|c|c|c|c|c|c|c|c|}
  \hline
k & 1&2 &3&4&5&6&7 \\
\hline
m & -17&-14 &-11&-8&-10&-13&-18  \\\hline
\end{tabular}
\end{center}
     \item  For $E_8$,

     \begin{center}
         \begin{tabular}{|c|c|c|c|c|c|c|c|c|}
  \hline
k & 1&2 &3&4&5&6&7&8 \\
\hline
m & -23&-17 &-13&-9&-11&-14&-19&-29  \\\hline
\end{tabular}
\end{center}
     \item  For $F_4$,

\begin{center}
         \begin{tabular}{|c|c|c|c| c|}
  \hline
k & 1&2 &3&4\\
\hline
m & -8&-5 &-7&-11   \\\hline
\end{tabular}
\end{center}

 \end{enumerate}

 \end{lem}

 \begin{prop}\label{F1F2}
   With the choice of $X$, $F_1$ and $F_2$ as in Table \ref{table1}, then we have
  \begin{enumerate}
      \item $F_1$ and $F_2$ are simple ACM bundles on $X.$
      \item $h^1(X,F_1\otimes F_2)\geq 4$ and $h^i(X,F_1\otimes F_2)=0$ for any $i\neq 1.$
      \item $h^i(X,F_1^\vee\otimes F_2^\vee)=0$ for any $i\in \mathbb{Z}$.
  \end{enumerate}
 \end{prop}

 \begin{proof}
 \begin{enumerate}
     \item  Since $F_1$ and $F_2$ are irreducible homogenenous vector bundles, then first statement follows from Corollary \ref{Class2} and Theorem 12.3 in \cite{ottaviani1995rational}.
     \item Since $F_1$ and $F_2$ are irreducible $P(\alpha_k)$-modules, then, by Proposition 10.5 in \cite{ottaviani1995rational}, $F_1\otimes F_2$ is completely reducible and can be decomposed into the direct sum of irreducible $P(\alpha_k)$-modules. %$\bigoplus_{l} E_{\mu_l}$.

 If $D(G)-\{k\}$ is a disconnected Dynkin diagram, then in terms of the choice in Table \ref{table1}, we find $F_1 =E_{u_1}$ and $F_2=E_{u_2}$
 where \[u_1=\lambda_1+\sum\limits_{j<k,j \text{ adjacent to }k}\lambda_j,\]
 and
 \[u_2=\left\{\begin{matrix}
 2\lambda_{k+1}-2\lambda_k,&\text{if}~X=F_4/P(\alpha_2),\\
 \lambda_{k+1}-2\lambda_k,&otherwise.
 \end{matrix}\right.\]
In this case, $F_1$ and $F_2$ comes from different simple parts of the semisimple part. Hence $F_1\otimes F_2=E_{u_1+u_2}$.
Because there is only one poistive root $\alpha_k$ such that $(u_1+u_2+\rho,\alpha_k)<0$ and  $(u_1+u_2+\rho,\alpha)>0$ for any positive root $\alpha\ne \alpha_k$,
  $u_1+u_2+\rho$ is regular of index 1. By the Borel--Bott--Weil Theorem,
 \[h^1(X,E_{u_1+u_2})=\dim G_{s_{\alpha_k}(u_1+u_2+\rho)-\rho}.\]
 Note that here $s_{\alpha_k}(u_1+u_2+\rho)=u_1+u_2+\rho-\frac{2(u_1+u_2+\rho,\alpha_k)}{(\alpha_k,\alpha_k)}\alpha_k=\lambda_1+\rho.$
 Then, by Weyl's formula (see \cite{ottaviani1995rational} Remark 10.18), it is not hard to get that $h^1(X,E_{u_1+u_2})\geq 4.$

 If $D(G)-\{k\}$ is a connected Dynkin diagram, then $F_1$ and $F_2$, as we choose in Table \ref{table1}, comes from the same simple part. Hence we use Magma which is based on the Klymik's formula (See \cite{de2000lie} Proposition 8.12.3) to get the decomposition of $F_1\otimes F_2$ as follows.

  \[F_1\otimes F_2=\left\{\begin{matrix}
 E_{-2\lambda_1+2\lambda_2+2\lambda_7}\oplus E_{-2\lambda_1+2\lambda_3}\oplus E_{-2\lambda_1+\lambda_2+\lambda_3+\lambda_7},&\text{if}~X=E_7/P(\alpha_1),\\
 E_{3\lambda_1-2\lambda_2+\lambda_3}\oplus
E_{2\lambda_1-2\lambda_2+\lambda_4},&\text{if}~X=E_n/P(\alpha_2),\\
  E_{-2\lambda_1+\lambda_3+3\lambda_4}\oplus E_{-2\lambda_1+\lambda_2+2\lambda_4}\oplus E_{-\lambda_1+3\lambda_4}\oplus E_{-\lambda_1+\lambda_3+\lambda_4},&\text{if}~X=F_4/P(\alpha_1).\\
 \end{matrix}\right.\]
Then $H^i(X,F_1\otimes F_2)=\bigoplus_{l} H^i(X,E_{\mu_l}).$

Now it suffices to show that either $\mu_l+\rho$ is regular of index 1 or $\mu_l+\rho$ is singular. We show this by find specific positive roots as Table \ref{table3}.
 \begin{center}
\makeatletter\def\@captype{table}\makeatother
    \centering

                 \begin{tabular}{|c|c|c|c|}
  \hline
$X$ &Weight $\mu_l$ of the tensor decomposition&$\alpha$& $(\mu_l+\rho,\alpha)$  \\
\hline
$E_7/P(\alpha_1)$ & $ -2\lambda_1+2\lambda_2+2\lambda_7 $  &  $\alpha_1+\alpha_3$ & 0 \\
  & $-2\lambda_1+2\lambda_3$  &  $\alpha_1$ & -1 \\
  & $ -2\lambda_1+\lambda_2+\lambda_3+\lambda_7 $  &  $\alpha_1$ &-1  \\ \hline

$E_n/P(\alpha_2)$ & $ 3\lambda_1-2\lambda_2+\lambda_3 $   &$\alpha_2+\alpha_4$&0 \\

  & $ 2\lambda_1-2\lambda_2+\lambda_4 $   &$\alpha_2$ &-1\\\hline

$F_4/P(\alpha_1)$ & $ -2\lambda_1+\lambda_3+3\lambda_4 $ & $ \alpha_1+\alpha_2$ &0  \\
  & $ -2\lambda_1+\lambda_2+2\lambda_4 $ & $ \alpha_1$  &-1 \\
  & $ -\lambda_1+3\lambda_4 $ & $\alpha_1$  &0 \\
  & $ -\lambda_1+\lambda_3+\lambda_4 $ & $ \alpha_1$  &0 \\\hline

\end{tabular}
    \caption{Killing forms with suitable poisitive roots}
    \label{table3}
    \end{center}
Then second statement follows from the computation and the Borel--Bott--Weil Theorem as before.

 \item Serre duality tells us that
 \[H^i(X,F_1^{\vee}\otimes F_2^{\vee})\simeq H^{\dim X-i}(X,F_1\otimes F_2\otimes \mathcal{O}_X(K_X))^*.\]
Hence it suffices to show that $h^i(X,F_1\otimes F_2\otimes \mathcal{O}_X(K_X))=0$ for any $i\in \mathbb{Z}.$

It is easy to see that \[F_1\otimes F_2\otimes \mathcal{O}_X(m)=\bigoplus_lE_{\mu_l+m\lambda_k},\]
where $\mathcal{O}_X(K_X)=\mathcal{O}_X(m)$ follows from Lemma \ref{canonical}.

It suffices to show that $\mu_l+m\lambda_k+\rho$ is singular by the Borel--Bott--Weil Theorem. This follows from Table \ref{table7}.
\\

 \end{enumerate}

      \begin{footnotesize}

\makeatletter\def\@captype{table}\makeatother
    \centering
 \begin{tabular}{|c|c|c| }

  \hline
$X$ &$\mu_l+m\lambda_k$ &$\alpha$ satisfying $(\mu_l+m\lambda_k+\rho,\alpha)=0$    \\
\hline

$E_n/P(\alpha_3)$ & $ 2\lambda_1-(2n-1)\lambda_3+\lambda_4 $&$\sum_{i=1}^n\alpha_i+\sum_{j=4}^{n-2}\alpha_j$    \\\hline

$E_n/P(\alpha_4)$ & $ \lambda_1+\lambda_2+\lambda_3+\lambda_5-(n+3)\lambda_4 $ &$\sum_{i=1}^{n-1}\alpha_i$  \\\hline

$E_n/P(\alpha_5)$& $  \lambda_1 +\lambda_4-(n+5)\lambda_5+\lambda_6 $ & $\sum_{i=1}^n\alpha_i+\alpha_4$     \\\hline

$E_n/P(\alpha_6)~(n\neq 6)$ &  $  \lambda_1 +\lambda_5-(n+8)\lambda_6+\lambda_7 $ & $\sum_{i=1}^n\alpha_i+\alpha_3+2\alpha_4+\alpha_5$        \\\hline
$E_8/P(\alpha_7)$ & $  \lambda_1 +\lambda_6-21\lambda_7+\lambda_8 $ & $\alpha_1+\alpha_2+\alpha_3+2\alpha_4+\alpha_5+\alpha_6+\alpha_7+\alpha_8$
\\\hline
$F_4/P(\alpha_2)$  & $ 2\lambda_1-7\lambda_2+2\lambda_3  $& $\alpha_1+ \alpha_2+2\alpha_3 $ \\\hline
$F_4/P(\alpha_3)$ & $\lambda_1+\lambda_2-9\lambda_3+\lambda_4$ & $\alpha_1+\alpha_2+\alpha_3$    \\\hline
$E_7/P(\alpha_1)$ & $ -19\lambda_1+2\lambda_2+2\lambda_7 $  &  $\alpha_1+2\alpha_2+2\alpha_3+3\alpha_4+2\alpha_5+2\alpha_6+\alpha_7$ \\
 & $-19\lambda_1+2\lambda_3$  &  $\alpha_1+2\alpha_2+2\alpha_3+4\alpha_4+3\alpha_5+2\alpha_6+\alpha_7$ \\
  & $ -19\lambda_1+\lambda_2+\lambda_3+\lambda_7 $  & $\alpha_1+2\alpha_2+2\alpha_3+3\alpha_4+3\alpha_5+2\alpha_6+\alpha_7$  \\ \hline

$E_6/P(\alpha_2)$ & $ 3\lambda_1-13\lambda_2+\lambda_3 $   &$\alpha_1+ \alpha_2+ 2\alpha_3+2\alpha_4+ \alpha_5+ \alpha_6$  \\
  & $2 \lambda_1-13\lambda_2+\lambda_4 $   &$\alpha_1+ \alpha_2+ 2\alpha_3+2\alpha_4+ 2\alpha_5+ \alpha_6 $ \\\hline
$E_7/P(\alpha_2)$ & $ 3\lambda_1-16\lambda_2+\lambda_3 $   &$\alpha_1+ \alpha_2+ 2\alpha_3+2\alpha_4+2\alpha_5+2\alpha_6 +\alpha_7$ \\
  & $ 2\lambda_1-16\lambda_2+\lambda_4 $   &$\alpha_1+ \alpha_2+ 2\alpha_3+2\alpha_4+2\alpha_5+2\alpha_6 +\alpha_7$  \\\hline
$E_8/P(\alpha_2)$ & $ 3\lambda_1-19\lambda_2+\lambda_3 $   &$\alpha_1+ \alpha_2+ 2\alpha_3+3\alpha_4+ 2\alpha_5+ 2\alpha_6 +2\alpha_7+\alpha_8$\\
   & $ 2\lambda_1-19\lambda_2+\lambda_4 $   &$\alpha_1+ \alpha_2+ 2\alpha_3+3\alpha_4+ 2\alpha_5+ 2\alpha_6 +2\alpha_7+\alpha_8$ \\\hline

$F_4/P(\alpha_1)$ & $ -10\lambda_1+\lambda_3+3\lambda_4 $ & $ \alpha_1+2\alpha_2+3\alpha_3+2\alpha_4$  \\
  & $ -10\lambda_1+\lambda_2+2\lambda_4 $ & $ \alpha_1+2\alpha_2+4\alpha_3+2\alpha_4$   \\
  & $ -9\lambda_1+3\lambda_4 $ & $ \alpha_1+2\alpha_2+4\alpha_3+2\alpha_4$   \\
  & $ -9\lambda_1+\lambda_3+\lambda_4 $ & $\alpha_1+2\alpha_2+4\alpha_3+2\alpha_4$  \\\hline

\end{tabular}
    \caption{Killing form with $\alpha$ such that $\mu_l+m\lambda_k+\rho$ singular}
    \label{table7}
  \end{footnotesize}

 \end{proof}

 In order to prove the representation type of $X$, we need introduce the definition of weakly equivalence.

\begin{define}
   Given two extensions $\mathcal{E},\mathcal{E}'\in Ext^1(\mathcal{G},\mathcal{F})$ we say they are weakly equivalent denoted by $\sim_w$ if there exist isomorphisms $\psi,\phi,\varphi$ such that the following diagram commutes

   \centerline{
   \xymatrix{
   0\ar[r]& \mathcal{F}\ar[r]\ar[d]^{\psi}&\mathcal{E}\ar[d]^\phi\ar[r]& \mathcal{G}\ar[d]^\varphi\ar[r]&0\\
      0\ar[r]& \mathcal{F}\ar[r]&\mathcal{E}'\ar[r] &\mathcal{G}\ar[r]&0.
   }
   }

\end{define}
The main tool of the proof is Proposition 5.1.3 in \cite{pons2010acm}. Let us review this proposition.

\begin{prop}\label{weakeq}
   Let $X$ be a projective variety over k and $\mathcal{F}_1,\dots,\mathcal{F}_{r+1}$, with $r\geq 1$, be simple coherent sheaves on $X$ (i.e., $\textup{Hom}(\mathcal{F}_i,\mathcal{F}_i)=k$) on X such that $\textup{Hom}(\mathcal{F}_i,\mathcal{F}_j)=0~(i\neq j).$ Denote
   \[U=\textup{Ext}^1(\mathcal{F}_{r+1},\mathcal{F}_1)-\{0\}\times\dots\times \textup{Ext}^1(\mathcal{F}_{r+1},\mathcal{F}_r)-\{0\}.\]

       Then a sheaf $E$ that comes up from an extension of $\mathcal{F}_{r+1}$ by $\bigoplus\mathcal{F}_i$ is simple of and only if $[\mathcal{E}]\in U$
        and given two extensions $[\mathcal{E}],[\mathcal{E}']\in U$ we have that \[\textup{Hom}(\mathcal{E},\mathcal{E}')\neq 0\Longleftrightarrow [\mathcal{E}]\sim_w[\mathcal{E}'].\]

   To be more precise, the simple coherent sheaves $\mathcal{E}$ coming up from an extension of $\mathcal{F}_{r+1}$ by $\bigoplus\mathcal{F}_i$ \[0\to \bigoplus^{r}_{i=1}\mathcal{F}_i\to \mathcal{E}\to \mathcal{F}_{r+1}\to 0\] are parametrized, up to isomorphism, by \[(U/\sim_w)\simeq \mathbb{P}(\textup{Ext}^1(\mathcal{F}_{r+1},\mathcal{F}_1))\times\dots\times\mathbb{P}(\textup{Ext}^1(\mathcal{F}_{r+1},\mathcal{F}_r)).\]
\end{prop}
Now we can show the representation type of $X.$

 \begin{cor}\label{reptype}

 The exceptional Grassmannian which is one of the following types: \[E_n/P(\alpha_k)~,F_4/(P(\alpha_l)),\]
 where $(n,k)\neq (6,1),(7,7),(8,1),(8,8)$ and $l\neq 4$ is of wild type.
 \end{cor}
 \begin{proof}
 We want to use Proposition \ref{weakeq} to construct $r$-dimensional families of non-isomorphic indecomposable ACM sheaves for any $r$. The proof is similar to Theorem 4.6 in \cite{costa2016homogeneous}.
By Proposition \ref{weakeq} and the fact that the extension of ACM bundles are ACM, it is sufficient to find $r+1$ simple ACM bundles $E_1,\dots, E_{r+1}$ satisfying $\textup{Hom}(E_i,E_j)=0 ~(i\neq j)$ and $\dim \textup{Ext}^1(E_{r+1},E_i)\geq2~(i=1,\dots,r).$

Let $F_1$ and $F_2$ be two vector bundles as in Proposition \ref{F1F2}.
Since $\dim \textup{Ext}^1(F_2^\vee,F_1)=h^1(X,F_1\otimes F_2)\geq 4$ by Proposition \ref{F1F2}, we have $\dim \mathbb{P}(\textup{Ext}^1(F_2^\vee,F_1))\geq 3$. Then we can choose a sequence of nontrivial non-weakly equivalent vector bundles $\{E_i\}$ in $\mathbb{P}(\textup{Ext}^1(F_2^\vee,F_1) $ by Proposition \ref{weakeq}. Note that the non-weakly equivalence implies that $\textup{Hom}(E_i,E_j)=0~(i\neq j)$. Nontriviality of $E_i$ implies that $ E_i $ are simple. %Also, since $F_1$ and $F_2$ are ACM bundles, then $ E_i $ are also ACM.
It remains to show that $\dim \textup{Ext}^1(E_i,E_j)\geq2~(i\neq j).$

To this end, consider the two exact sequences defining $E_i$ and $E_j~(i\neq j)$ respectively,
\begin{equation}\label{equation1}
0\to F_1\to E_i\to F_2^\vee\to0,
\end{equation}
\begin{equation}\label{equation2}
0\to F_1\to E_j\to F_2^\vee\to0.
\end{equation}

 Applying $\textup{Hom}(E_i,-)$ to (\ref{equation2}),  we obtain the long exact sequence ($\textup{Hom}(E_i,E_j)=0~(i\neq j)$)
 \[0\to\textup{Hom} (E_i,F_2^\vee)\to \textup{Ext}^1(E_i,F_1)\xrightarrow{\tau}\textup{Ext}^1(E_i,E_j)\to \cdots.\]
Then
 \begin{equation}\label{eq1}
  \dim \textup{Ext}^1(E_i,E_j)\geq \dim \textup{Im}(\tau)=\dim \textup{Ext}^1(E_i,F_1)-\dim \textup{Hom}(E_i,F_2^\vee).
\end{equation}

 Applying $\textup{Hom}(-,F_2^\vee)$ to (\ref{equation1}), we get:
\[0\to \textup{Hom}(F_2^\vee,F_2^\vee)\to \textup{Hom}(E_i,F_2^\vee)\to \textup{Hom}(F_1,F_2^\vee)\to \cdots. \]
Since $F_2$ is simple and $\dim \textup{Hom}(F_1,F_2^\vee)=h^0(X,F_1^\vee\otimes F_2^\vee))=0$ (see Proposition \ref{F1F2}), we have
\begin{equation}\label{eq2}
 \dim\textup{Hom}(E_i,F_2^\vee)=\dim\textup{Hom}(F_2^\vee,F_2^\vee)=1.
\end{equation}

On the other hand, applying $\textup{Hom}(-,F_1)$ to (\ref{equation1}), we obtain:
\[0\to
\textup{Hom}(F_2^\vee,F_1) \to \textup{Hom}(E_i,F_1)\to \textup{Hom}(F_1,F_1)\xrightarrow{\eta_1}\textup{Ext}^1(F_2^\vee,F_1)\to\textup{Ext}^1(E_i,F_1)\to \cdots .\]
Then
\[\dim \textup{Ext}^1(E_i,F_1)\geq\dim \textup{Ext}^1(F_2^\vee,F_1)-\dim \textup{Im}(\eta_1)\geq \textup{Ext}^1(F_2^\vee,F_1)-\dim \textup{Hom}(F_1,F_1).\]
Since $F_1$ is simple and $\dim \textup{Ext}^1(F_2^\vee,F_1)=h^1(X,F_1\otimes F_2)\geq 4$ (see Proposition \ref{F1F2}), we have $\dim \textup{Ext}^1(E_i,F_1)\geq3$. This together with $(\ref{eq1})$ and $(\ref{eq2})$ give us
$\dim \textup{Ext}^1(E_i,E_j)\geq2.$

 \end{proof}

 \appendix
{\centering\section*{Appendix}}

In fact, with the help of a computer, we can list all irreducible homogenenous ACM bundles on exceptional Grassmannians. The algorithm is similar to the proof of Example \ref{classG}. However, there are so many ACM bundles on $E_n/P(\alpha_k)$~($k< n$) which are hard to list. For example, $E_6/P(\alpha_4)$ has 830 initialized irreducible homogeneous ACM bundles. Hence we only list the rest cases below.

 Let $E_\lambda$ be an initialized irreducible homogenenous vector bundle with highest weight $\lambda=b_1\lambda_1+\dots+b_n\lambda_n$.
 Then $E_\lambda$ is an ACM bundle if and only if the coefficient $b_i$ is of the the following form.
\begin{enumerate}
\item For $G_2/P(\alpha_1)$,
     \begin{center}

                   \centering
\begin{tabular}{|c|c| }
  \hline
  &	1		\\
  \hline
$b_1$&0	\\
$b_2$&0\\
 \hline

\end{tabular}
\end{center}
\item For $G_2/P(\alpha_2)$,
     \begin{center}

                   \centering
\begin{tabular}{|c|c|c|c| }
  \hline
 &	1	&	2	&	3	\\
  \hline
$b_1$&0&1&2	\\
$b_2$&0&0&0\\
 \hline

\end{tabular}
\end{center}
    \item    For $F_4/P(\alpha_1)$,
     \begin{center}

                   \centering
\begin{tabular}{|c|c|c|c|c|c|c|c|c| c|c| }
  \hline
 &	1	&	2	&	3	&	4	&	5	&	6	&	7	&	8	&	9	&	10	\\
  \hline
$b_1$	&	0	&	0	&	0	&	0	&	0	&	0	&	0	&	0	&	0	&	0	\\
$b_2$	&	0	&	0	&	0	&	0	&	0	&	0	&	0	&	0	&	0	&	0	\\
$b_3$	&	1	&	0	&	1	&	0	&	1	&	0	&	1	&	0	&	1	&	0	\\
$b_4$	&	0	&	1	&	1	&	2	&	2	&	3	&	3	&	4	&	5	&	0	\\ \hline

\end{tabular}
\end{center}

\item  For $F_4/P(\alpha_2)$,

   \begin{center}

                   \centering
\begin{tabular}{|c|c|c|c|c|c|c|c|c|c|c|c|c|c|c|c|c|}
  \hline
  &1&2&3&4&5&6&7&8&9&10 &11&12&13&14&15&16\\
\hline
$b_1$	&	1	&	2	&	3	&	4	&	5	&	0	&	1	&	2	&	3	&	4	&	5	&	0	&	0	&	0	&	1	&	2\\
$b_2$	&	0	&	0	&	0	&	0	&	0	&	0	&	0	&	0	&	0	&	0	&	0	&	0	&	0	&	0	&	0	&	0\\
$b_3$	&	0	&	0	&	0	&	0	&	0	&	1	&	1	&	1	&	1	&	1	&	1	&	2	&	3	&	0	&	0	&	0\\
$b_4$	&	0	&	0	&	0	&	0	&	0	&	0	&	0	&	0	&	0	&	0	&	0	&	0	&	0	&	1	&	1	&	0\\

\hline
&17&18 &19 &20 &21&22&23&24&25&26&27&28&29&30&31&32 \\
\hline
$b_1$	&	3	&	4	&	5	&	6	&	0	&	1	&	2	&	3	&	4	&	5	&	6	&	7	&	8	&	0	&	0	&	0	\\
$b_2$	&	0	&	0	&	0	&	0	&	0	&	0	&	0	&	0	&	0	&	0	&	0	&	0	&	0	&	0	&	0	&	0	\\
$b_3$	&	0	&	0	&	0	&	0	&	1	&	1	&	1	&	1	&	1	&	1	&	1	&	1	&	1	&	2	&	3	&	0	\\
$b_4$	&	1	&	1	&	1	&	1	&	1	&	1	&	1	&	1	&	1	&	1	&	1	&	1	&	1	&	1	&	1	&	2	\\ \hline

&33&34 &35 &36 &37&38&39&40& &  & & &&&& \\\hline
$b_1$	&		0	&	0	&	0	&	0	&	0	&	0	&	0	&	0	& &  & & &&&& \\
$b_2$	&		0	&	0	&	0	&	0	&	0	&	0	&	0	&	0	& &  & & &&&& \\
$b_3$	&		1	&	2	&	3	&	0	&	1	&	2	&	3	&	0	& &  & & &&&& \\
$b_4$	&		2	&	2	&	2	&	3	&	3	&	3	&	3	&	0	& &  & & &&&& \\ \hline

\end{tabular}
\end{center}

\item   For $F_4/P(\alpha_3)$,
 \begin{center}

                   \centering
\begin{tabular}{|c|c|c|c|c|c|c|c|c|c|c|c|c| }
  \hline
 	&1&2&3&4&5&6&7&8&9&10 &11&12 	\\
  \hline
$b_1$	&	1	&	0	&	1	&	0	&	0	&	0	&	0	&	0	&	0	&	0	&	0	&	0	\\
$b_2$	&	0	&	1	&	1	&	0	&	0	&	0	&	0	&	0	&	0	&	0	&	0	&	0	\\
$b_3$	&	0	&	0	&	0	&	0	&	0	&	0	&	0	&	0	&	0	&	0	&	0	&	0	\\
$b_4$	&	0	&	0	&	0	&	1	&	2	&	3	&	4	&	5	&	6	&	7	&	8	&	0	\\\hline

\end{tabular}
\end{center}
\item   For $F_4/P(\alpha_4)$,
\begin{center}

                   \centering
\begin{tabular}{|c|c|c|c| }
  \hline
  &1&2&3  \\
  \hline
$b_1$	&	1	&	2	&	0	\\ \hline
$b_2$	&	0	&	0	&	0	\\ \hline
$b_3$	&	0	&	0	&	0	\\ \hline
$b_4$	&	0	&	0	&	0	\\ \hline

\end{tabular}
\end{center}
\item   For $E_6/P(\alpha_6)$,
 \begin{center}

                   \centering
\begin{tabular}{|c|c|c|c|c|c|c|c| }
  \hline
 &$b_1$ &$b_2$&$b_3$&$b_4$&$b_5$&$b_6$ \\
  \hline
 	1&0	&	0	&	0	&	0	&	0&0	\\ \hline
2 &1	&	0	&	0	&	0	&	0	&	0	\\ \hline
3 & 2	&	0	&	0	&	0	&	0	&	0	\\ \hline
4 &3	&	0	&	0	&	0	&	0	&	0	\\ \hline
5 &0	&	0	&	0	&	1	&	0&	0	\\ \hline
6 &1	&	0	&	1	&	0	&	0	&	0	\\ \hline
7 & 2	&	0	&	1	&	0	&	0	&	0	\\ \hline
8 & 3	&	0	&	1	&	0	&	0	&	0	\\ \hline
\end{tabular}
\end{center}

\item   For $E_7/P(\alpha_7)$,
 \begin{center}

                   \centering
\begin{tabular}{|c|c|c|c|c|c|c|c| }
  \hline
 &$b_1$ &$b_2$&$b_3$&$b_4$&$b_5$&$b_6$&$b_7$ \\
  \hline
1&0	&	1	&	0	&	0	&	0	&	0	&	0	\\ \hline
2&0	&	2	&	0	&	0	&	0	&	0	&	0	\\ \hline
3&0	&	0	&	0	&	0	&	0	&	0	&	0	\\ \hline
\end{tabular}
\end{center}

\item   For $E_8/P(\alpha_8)$,
 \begin{center}

                   \centering
\begin{tabular}{|c|c|c|c|c|c|c|c|c| }
  \hline
 &$b_1$ &$b_2$&$b_3$&$b_4$&$b_5$&$b_6$&$b_7$ &$b_8$\\
  \hline
1&1	&	0	&	0	&	0	&	0	&	0	&	0	&	0	\\ \hline
2&2	&	0	&	0	&	0	&	0	&	0	&	0	&	0	\\ \hline
3&3	&	0	&	0	&	0	&	0	&	0	&	0	&	0	\\ \hline
4&4	&	0	&	0	&	0	&	0	&	0	&	0	&	0	\\ \hline
5&5	&	0	&	0	&	0	&	0	&	0	&	0	&	0	\\ \hline
6&0	&	1	&	0	&	0	&	0	&	0	&	0	&	0	\\ \hline
7&1	&	1	&	0	&	0	&	0	&	0	&	0	&	0	\\ \hline
8&2	&	1	&	0	&	0	&	0	&	0	&	0	&	0	\\ \hline
9&3	&	1	&	0	&	0	&	0	&	0	&	0	&	0	\\ \hline
10&4	&	1	&	0	&	0	&	0	&	0	&	0	&	0	\\ \hline
11&0	&	2	&	0	&	0	&	0	&	0	&	0	&	0	\\ \hline
12&0	&	0	&	1	&	0	&	0	&	0	&	0	&	0	\\ \hline
13&1	&	0	&	1	&	0	&	0	&	0	&	0	&	0	\\ \hline
14&2	&	0	&	1	&	0	&	0	&	0	&	0	&	0	\\ \hline
15&3	&	0	&	1	&	0	&	0	&	0	&	0	&	0	\\ \hline
16&4	&	0	&	1	&	0	&	0	&	0	&	0	&	0	\\ \hline
17&5	&	0	&	1	&	0	&	0	&	0	&	0	&	0	\\ \hline
18&0	&	0	&	0	&	0	&	0	&	0	&	0	&	0	\\ \hline
\end{tabular}
\end{center}
\end{enumerate}

\bibliography{ref}

\end{document}